\documentclass[]{amsart}
\usepackage{mathrsfs}
\usepackage{color}
\usepackage{amsmath}
\usepackage{amsfonts}
\usepackage{amssymb}

 \newtheorem{Theorem}{Theorem}[section]
 \newtheorem{Corollary}[Theorem]{Corollary}
 \newtheorem{Lemma}[Theorem]{Lemma}
 \newtheorem{Proposition}[Theorem]{Proposition}

 \newtheorem{Remark}[Theorem]{Remark}

 \numberwithin{equation}{section}

\begin{document}

\title
 {Effectiveness of strong openness property in $L^p$}

\author{Qi'an Guan, Zheng Yuan}
\address{Qi'an Guan: School of Mathematical Sciences,
Peking University, Beijing, 100871, China. }
\email{guanqian@math.pku.edu.cn}

\address{Zheng Yuan: School of Mathematical Sciences,
Peking University, Beijing, 100871, China.}
\email{zyuan@pku.edu.cn}

\thanks{}



\date{\today}

\dedicatory{}

\commby{}


\begin{abstract}
In this article, we obtain an effectiveness result of strong openness property in $L^p$ with some applications.
\end{abstract}

\maketitle

\section{Introduction}

The concept of multiplier ideal sheaf helps to translate $L^2$ estimates into algebraic conditions,
which plays an important role and was widely discussed in several complex variables, complex geometry and algebraic geometry
(see e.g. \cite{tian87,Nadel90,siu96,siu00,demailly-note2000,DEL00,D-K01,D-P03,siu05,siu09,demailly2010}).

Recall the definition of the multiplier ideal sheaf $\mathcal{I}(\varphi)$ (see \cite{demailly-note2000,demailly2010}):
a germ of holomorphic function $(f,z)\in\mathcal{I}(\varphi)_{z}$ if and only if $|f|^{2}e^{-\varphi}$ is integrable near $z$,
where $\varphi$ is a plurisubharmonic function on a complex manifold $X$.

The strong openness property for multiplier ideal sheaves i.e. $\mathcal{I}(\varphi)=\mathcal{I}_{+}(\varphi):=\cup_{q>1}\mathcal{I}(q\varphi)$, was conjectured by Demailly (see \cite{demailly-note2000,demailly2010}) and proved by Guan-Zhou \cite{GZopen-c}
(see also \cite{Hiep14,Lempert14}, 2-dimensional case was proved by Jonsson-Must\c{a}t\u{a} \cite{JM12}).

There is an important special case of the strong openness property, which was called the openness property i.e. if $\mathcal{I}(\varphi)=\mathcal{O}$, then $\mathcal{I}(\varphi)=\mathcal{I}_{+}(\varphi)$.
The openness property was conjectured by Demailly-Koll\'{a}r \cite{D-K01} and proved by Berndtsson \cite{berndtsson13} (2-dimensional case was proved by Favre-Jonsson \cite{FM05j}).

We would like to recall that Berndtsson's proof of the openness property established the effectiveness result of the openness property \cite{berndtsson13}.
Stimulated by the effectiveness in Berndtsson's proof of the openness property and continuing Guan-Zhou's proof of the strong openness property,
Guan-Zhou \cite{GZopen-effect} established the effectiveness result of the strong openness property.

Let $D\subset\mathbb{C}^n$ be a pseudoconvex domain containing the origin $o$.
Let $F$ be a holomorphic function on $D$ and let $\varphi$ be a negative plurisubharmonic function on $D$. Recall a notation in \cite{GZopen-effect} that
$$K_{\varphi,F}(o):=\frac{1}{\inf\{\int_{D}|\tilde F|^2:(\tilde F-F,o)\in\mathcal{I}_+(2c_{o}^{F}(\varphi)\varphi)_o\,\&\,\tilde F\in\mathcal{O}(D)\}},$$
where $c_o^{F}(\varphi):=\sup\{c\geq0:|F|^2e^{-2c\varphi}$ is $L^1$ on a neighborhood of $o$$\}$ is the jumping number (see \cite{JM13}).
Let $$\theta(q):=(\frac{1}{(q-1)(2q-1)})^{\frac{1}{q}},$$
where $q\in(1,+\infty)$.
\begin{Theorem}\label{gz-eff}
	\cite{GZopen-effect} Let $C_1$ and $C_2$ be two positive constants. We consider the set of the pairs $(F,\varphi)$ satisfying
	
	$(1)$ $\int_{D}|F|^2e^{-\varphi}\leq C_1$;
	
	$(2)$ $(K_{\varphi,F})^{-1}(o)\geq C_2$.
	
	Then for any $q>1$ satisfying $$\theta(q)>\frac{C_1}{C_2},$$ we have $$(F,o)\in\mathcal{I}(q\varphi)_o.$$
\end{Theorem}
When $F\equiv1$, Theorem \ref{gz-eff} implies Berndtsson's effectiveness result of the openness property (\cite{berndtsson13}, see also \cite{GZopen-effect}).

By considering the minimal
$L^2$ integrals related to a multiplier ideal sheaf $\mathcal{I}(\varphi)$ on the sublevel sets of the weight $\varphi$, Guan \cite{guan_sharp} obtained a sharp version of Theorem \ref{gz-eff},
i.e.
$$\theta(q)=\frac{q}{q-1},$$
and presented a concavity property of the minimal $L^{2}$ integrals.
After that, Guan \cite{guan_general concave} (see also \cite{GMconcave}) generalized the above concavity property.

In \cite{Fo15}, Forn{\ae}ss established the following strong openness property in $L^p$:

\emph{Let $F$ be a holomorphic function on a domain $D\subset\mathbb{C}^n$ containing the origin $o$, $\varphi$ a plurisubharmonic function on $D$ and $p\in(0,+\infty)$. If $|F|^pe^{-\varphi}$ is $L^1$ on a neighborhood of $o$, then there exists $q>1$ such that $|F|^pe^{-q\varphi}$ is $L^1$ on a neighborhood of $o$.}

In this article, we obtain an effectiveness result of the strong openness property in $L^p$ by using the general concavity in \cite{guan_general concave}.

\subsection{Main result}

Let $D\subset\mathbb{C}^n$ be a pseudoconvex domain containing the origin $o$.
Let $F$ be a holomorphic function on $D$,
and let $\varphi$ be a negative plurisubharmonic function on $D$. Let $p\in(0,+\infty)$, and denote that $F_1=F^{\lceil\frac{p}{2}\rceil}$, $\varphi_1=(2\lceil\frac{p}{2}\rceil-p)\log|F|$, where $\lceil m\rceil:=\min\{n\in \mathbb{Z}:n\geq m\}$. Then $\int_{D}|F|^pe^{-\varphi}=\int_{D}|F_1|^2e^{-(\varphi_1+\varphi)}$.

Take $c_{o,p}^{F}(\varphi):=\sup\{c\geq0:|F|^pe^{-2c\varphi}$ is $L^1$ on a neighborhood of $o$$\}$. Especially, when $p=2$, $c_{o,p}^{F}(\varphi)$  degenerates to the jumping number $c_o^{F}(\varphi)$. If $c_{o,p}^{F}(\varphi)<+\infty$, we generalize   $K_{\varphi,F}$ as follows:
$$K^{(p)}_{\varphi,F,a}(o):=\frac{1}{\inf\{\int_D|\tilde F|^2e^{-\varphi_1-(1-a)\varphi
}:(\tilde F-F_1,o)\in \mathcal{I}(\varphi_1+2c_{o,p}^{F}(\varphi)\varphi)_o \, \& \,\tilde F\in\mathcal{O}(D)\}},$$
where $a\in(0,+\infty)$.
Especially, when $p=2$ and $a=1$, $K^{(p)}_{\varphi,F,a}(o)$ degenerates to $K_{\varphi,F}$. Note that $K^{(p)}_{\varphi,F,a}(o)<+\infty$ (see Appendix).

We obtain an  effectiveness result of the strong openness property in $L^{p}$.
\begin{Theorem}
	\label{c:lp}
	Let $C_1$ and $C_2$ be two positive constants. If there exists $a>0$, such that
	
	$(1)$ $\int_{D}|F|^pe^{-\varphi}\leq C_1$;
	
	$(2)$ $(K_{\varphi,F,a}^{(p)})^{-1}(o)\geq C_2$.
	
	Then for any $q>1$ satisfying $$\theta_a(q)>\frac{C_1}{C_2},$$ we have
$|F|^pe^{-q\varphi}$ is $L^1$ on a neighborhood of $o$, where $\theta_a(q)=\frac{q+a-1}{q-1}$.	
\end{Theorem}

	Let $D=\Delta$, $F=z$ and $\varphi=\frac{p+2}{q}\log|z|$, then $c_{o,p}^{F}(\varphi)=\frac{q}{2}$. By some calculations, we have $\int_{D}|F|^{p}e^{-\varphi}=\frac{2q\pi}{(q-1)(p+2)}$ and $K_{\varphi,F,a}^{(p)}(o)=\frac{(q+a-1)(p+2)}{2q\pi}$, then $K_{\varphi,F,a}^{(p)}(o)\int_{D}|F|^pe^{-\varphi}=\frac{q+a-1}{q-1}$, which implies the sharpness of Theorem \ref{c:lp}.

When $F\equiv1$, $K_{D,(1-a)\varphi}(o)\geq K_{\varphi,F,a}^{(p)}(o)$, where $K_{D,(1-a)\varphi}$ is the Bergman kernel with weight $e^{-(1-a)\varphi}$ on $D$. Note that $K_{D,(1-a)\varphi}(o)<+\infty$ (see Appendix). Theorem \ref{c:lp} implies the following effectiveness result of the openness property.
\begin{Corollary}
	\label{c:o-effect}
	Let $C_1$ and $C_2$ be two positive constants.  If there exists $a>0$, such that
	
	$(1)$ $\int_{D}e^{-\varphi}\leq C_1$;
	
	$(2)$ $(K_{D,(1-a)\varphi})^{-1}(o)\geq C_2$.
	
	Then for any $q>1$ satisfying\begin{equation}
		\label{eq:210704d}\theta_a(q)>\frac{C_1}{C_2},	\end{equation}
		we have
$e^{-q\varphi}$ is $L^1$ on a neighborhood of $o$, where $\theta_a(s)=\frac{q+a-1}{q-1}$.	
\end{Corollary}

 When $D=\Delta$ and $\varphi=\frac{2}{q}\log|z|$. By some calculations, we have $\int_{D}e^{-\varphi}=\frac{q\pi}{q-1}$ and $K_{D,(1-a)\varphi}(o)=\frac{q+a-1}{q\pi}$, then $K_{D,(1-a)\varphi}(o)\int_{D}e^{-\varphi}=\frac{q+a-1}{q-1}$, which implies the sharpness of Corollary \ref{c:o-effect}.

\subsection{Applications: more precise versions of some known results}

 In this section, using Theorem \ref{c:lp} and Corollary \ref{c:o-effect}, we give more precise versions of some known effectiveness results of the strong openness property in $L^2$ and the openness property.

When $p=2$,  Theorem \ref{c:lp} is the following effectiveness result of the strong openness property in $L^{2}$.

\begin{Corollary}
\label{c:l2}
Let $C_1$ and $C_2$ be two positive constants. If there exists $a>0$, such that
	
	$(1)$ $\int_{D}|F|^2e^{-\varphi}\leq C_1$;
	
	$(2)$ $(K_{\varphi,F,a}^{(2)})^{-1}(o)\geq C_2$.
	
	Then for any $q>1$ satisfying
	\begin{equation}
		\label{eq:210704b}\theta_a(q)>\frac{C_1}{C_2},
	\end{equation}
	 we have
$|F|^2e^{-q\varphi}$ is $L^1$ on a neighborhood of $o$, where $\theta_a(s)=\frac{q+a-1}{q-1}$.	
\end{Corollary}

Note that $K_{\varphi,F,1}^{(2)}=K_{\varphi,F}$. When $a=1$, Corollary \ref{c:l2} is the sharp version of Theorem \ref{gz-eff} in \cite{guan_sharp}:

\emph{Let $C_1$ and $C_2$ be two positive constants. We consider the set of the pairs $(F,\varphi)$ satisfying: $\int_{D}|F|^2e^{-\varphi}\leq C_1$ and $(K_{\varphi,F})^{-1}(o)\geq C_2$.
Then for any $q>1$ satisfying
	\begin{equation}
		\label{eq:20210602b}\theta(q)>\frac{C_1}{C_2},
	\end{equation}
	 we have
$|F|^2e^{-q\varphi}$ is $L^1$ on a neighborhood of $o$, where $\theta(q)=\frac{q}{q-1}$.}

It is clear that $\frac{q}{q-1}>(\frac{1}{(q-1)(2q-1)})^{\frac{1}{q}}$ for any $q>1$, then Corollary \ref{c:l2} implies Theorem \ref{gz-eff}.
The following remark shows that Corollary \ref{c:l2} is more precise than the sharp version of Theorem \ref{gz-eff} in \cite{guan_sharp}.

\begin{Remark}
Let $D$ be $\Delta^2\in\mathbb{C}^2$. Let $F=z_1+z_2$ and $\varphi=\log|z_1|$, then $c_o^{F}(\varphi)=1$. By some calculations, we have $\int_{D}|F|^2e^{-\varphi}=\frac{5}{3}\pi^2$, and $K_{\varphi,F,a}^{(2)}(o)=\frac{a+1}{\pi^2}$.
Then inequality \eqref{eq:210704b} implies $q<\frac{8}{5}$, and
inequality \eqref{eq:20210602b} implies $q<\frac{10}{7}$.
\end{Remark}

When $F\equiv1$, the sharp version of Theorem \ref{gz-eff} in \cite{guan_sharp} implies the following result (\cite{guan_sharp}):

 \emph{Let $C_1$ and $C_2$ be two positive constants. We consider the set of $\varphi$ satisfying:
	 $\int_{D}e^{-\varphi}\leq C_1$ and $(K_{D})^{-1}(o)\geq C_2$.
	Then for any $q>1$ satisfying
	\begin{equation}
		\label{eq:210704c}\theta(q)>\frac{C_1}{C_2},
	\end{equation}
	 we have
$e^{-q\varphi}$ is $L^1$ on a neighborhood of $o$, where $\theta(q)=\frac{q}{q-1}$ and $K_D$ is the Bergman kernel on $D$.}

The above result is the case $a=1$ of Corollary \ref{c:o-effect}.

It is known that Guan-Zhou's effectiveness result of strong openness property (\cite{GZopen-effect}) implies Berndtsson's effectiveness result of the openness property (\cite{berndtsson13}, see also \cite{GZopen-effect}).
Note that $\frac{q}{q-1}>(\frac{1}{(q-1)(2q-1)})^{\frac{1}{q}}$ for any $q>1$, then Corollary \ref{c:o-effect} is a general version of Berndtsson's effectiveness result of the openness property.

The following remark shows that Corollary \ref{c:o-effect} is more precise than the above result (\cite{guan_sharp}) with condition \eqref{eq:210704c}.
\begin{Remark}
Let $D$ be $\Delta^2\in\mathbb{C}^2$, and let $\varphi=\log|z_1|+\log|z_2|$. By some calculations, we have $\int_{D}e^{-\varphi}=4\pi^2$, and $K_{D,(1-a)\varphi}(o)=\frac{(a+1)^2}{4\pi^2}$.
Then inequality \eqref{eq:210704d} implies $q<\frac{3}{2}$, and
inequality \eqref{eq:210704c} implies $q<\frac{4}{3}$.
\end{Remark}

\section{Proof of Theorem \ref{c:lp}}

Firstly, we recall two lemmas which will be used in the proof of Propersition \ref{p}.
\begin{Lemma}\cite{guan_sharp}
	\label{l:1}
	Let $F$ be a holomorphic funtion on pseudoconvex domain $D$. Let $\psi$ be a negative plurisubharmonic function on $D$, and let $\varphi'$ be a  plurisubharmonic function on $D$. Assume that $\int_{D}|F|^2e^{-\varphi'}<+\infty$. Then
	\begin{equation}
		\label{eq:20210529b}
		\int_{D}|F|^2e^{-\varphi'}=\int_{-\infty}^{+\infty}(\int_{\{\psi<-t\}}|F|^2e^{-\varphi'+\psi})e^tdt.
	\end{equation}
\end{Lemma}
\begin{proof}
	It is clear that the lemma directly follows from the basic formula \cite{GBF}	\begin{equation}
	\label{eq:l62}
		\int_Xfd\mu=\int_0^{+\infty}\mu(\{x\in X:f(x)>l\})dl
	\end{equation}
  for nonnegative measurable function $f:X\rightarrow[0,+\infty)$ with $X=D$, $f=e^{-\psi}$, and $d\mu=|F|^2e^{-\varphi'+\psi}dV_{2n}$, where $dV_{2n}$ is the Lebesgue measure in $\mathbb C^n$.

  Next, we will prove the equality \eqref{eq:l62}.
  \begin{displaymath}
  	\begin{split}
  		\int_Xfd\mu &=\int_X(\int_0^{f(x)}dl)d\mu
  		\\&=\int_X(\int_0^{+\infty} \theta_f(l,x) dl)d\mu
  		\\&=\int_0^{+\infty}(\int_X\theta_f(l,x)d\mu)dl
  		\\&=\int_0^{+\infty}\mu(\{x\in X:f(x)>l\})dl,  		\end{split}
  \end{displaymath}
  where
    $\theta_f(l,x) =
  \left\{ \begin{array}{cc}
  1 & \textrm {if $l<f(x)$}\\
  0 & \textrm{if $l\geq f(x)$}  	
  \end{array} \right.$
  is a function defined on $\mathbb R\times X$. Then the equality \eqref{eq:l62} has thus been completed.	
\end{proof}

We recall some notations and definitions in \cite{guan_general concave}. Let $\tilde\psi<0$ be a plurisubharmonic function on $D$ satisfying $\tilde\psi(o)=-\infty$, and let $\tilde \varphi$ be a Lebesgue measurable function on $D$ such that $\tilde\varphi+\tilde\psi$ is a plurisubharmonic function on $D$. We call a positive smooth function $c$  on $(0,+\infty)$ in class $\mathcal P_0$ if the following three statements hold:

(1) $\int_{0}^{+\infty}c(t)e^{-t}<+\infty$;

(2) $c(t)e^{-t}$ is decreasing with respect to $t$;

(3) for any compact subset $K\subset D$, $e^{-\tilde\varphi}c(-\tilde\psi)$ has a positive lower bound on $K$.

Define a function $G(t;c):[0,+\infty)\rightarrow [0,+\infty]$ ($G(t)$ for short without misunderstanding) by
\begin{equation*}
\inf\{\int_{\{\psi<-t\}}|\tilde{f}|^{2}e^{-\tilde\varphi}c(-\tilde\psi):(\tilde{f}-F,o)\in \mathcal I(\tilde\varphi+\tilde\psi)_o\,\&\,\tilde{f}\in \mathcal O(\{\psi<-t\})\},
\end{equation*}
 where $c\in\mathcal P_0$.

We will use the following concavity of $G(t)$  in the proof of Proposition \ref{p}. Let $h(t)=\int_{t}^{+\infty}c(t_{1})e^{-t_{1}}dt_{1}$.

\begin{Lemma}
\label{l:concave}(\cite{guan_general concave}, see also \cite{GMconcave})
Let $c\in\mathcal P_0$. If $G(0)<+\infty$, then $G(h^{-1}(r))$ is concave with respect to $r\in(0,\int_{0}^{+\infty}c(t)e^{-t}dt]$.	
\end{Lemma}

Next, we prove  two propositions. Let $I$ be an ideal in $\mathcal{O}_o$, and take
$$C_{F,I,\Phi_1-\Phi_2}(D):=\inf\{\int_D|\tilde F|^2e^{-\Phi_1+\Phi_2
}:(\tilde F-F,o)\in  I \, \& \,\tilde F\in\mathcal{O}(D)\},$$
where $\Phi_1\not\equiv-\infty$ and $\Phi_2\not\equiv-\infty$ are plurisubharmonic functions on $D$. $C_{F,I,\Phi_1-\Phi_2}(D)=0$ if and only if $(F,o)\in I$ (see Appendix).\begin{Proposition}
\label{p}
	Let $F$ be a holomorphic funtion on pseudoconvex domain $D$. Let  $\psi$  be a negative plurisubharmonic function on $D$ satisfying $\psi(o)=-\infty$,  and let $\varphi'$ be a  plurisubharmonic function on $D$. Then for any $q'>1$, we have
	\begin{equation}
		\label{eq:20210529a}
		\int_{D}|F|^2e^{-\varphi'}\geq\frac{q'}{q'-1}C_{F,\mathcal{I}(\varphi'+(q'-1)\psi)_o,\varphi'-\psi}(D).
	\end{equation}
\end{Proposition}
\begin{proof}
	It suffices to consider the situation that $\int_{D}|F|^2e^{-\varphi'}<+\infty$. For any $l\in \mathbb{N}^+$, we can find an increasing smooth function $c_l(t)$ on $(0,+\infty)$, such that $c_l(t)=1$ when $t\in(0,l)$, $c_l(t)=e^{\frac{t-l}{q'}}$ when $t>l+1$, and $c(t)e^{-t}$ is decreasing with respect to $t$. Let $\tilde\varphi=\varphi'-\psi$, $\tilde\psi=q'\psi$, and it is clear that $c_l\in\mathcal P_0$. Then Lemma \ref{l:concave} implies that
	\begin{equation}
		\label{eq:20210529c}
		\begin{split}
		\int_{\{q'\psi<-q't\}}|F|^2e^{-\varphi'+\psi}c_l(-q'\psi)\geq&\int_{q't}^{+\infty}c_l(s)e^{-s}ds\frac{G(0,c_l)}{\int_0^{+\infty}c_l(s)e^{-s}ds}	\\
		\geq&\int_{q't}^{+\infty}c_l(s)e^{-s}ds\frac{C_{F,\mathcal{I}(\varphi'+(q'-1)\psi)_o,\varphi'-\psi}(D)}{\int_0^{+\infty}c_l(s)e^{-s}ds}	.
		\end{split}	\end{equation}
	As $\int_{D}|F|^2e^{-\varphi'}<+\infty$, let $l\rightarrow+\infty$, then it follows from the dominated convergence theorem that
	\begin{equation}
		\label{eq:20210529d}
		\int_{\{q'\psi<-q't\}}|F|^2e^{-\varphi'+\psi}\geq e^{-q't}C_{F,\mathcal{I}(\varphi'+(q'-1)\psi)_o,\varphi'-\psi}(D).
			\end{equation}
Combining Lemma \ref{l:1}, we have
\begin{equation}
	\label{eq:20210529e}
	\begin{split}
		\int_{D}|F|^2e^{-\varphi'}=&\int_{-\infty}^{+\infty}(\int_{\{\psi<-t\}}|F|^2e^{-\varphi'+\psi})e^{t}dt\\
		=&\int_{0}^{+\infty}(\int_{\{\psi<-t\}}|F|^2e^{-\varphi'+\psi})e^{t}dt+\int_{-\infty}^{0}(\int_{\{\psi<-t\}}|F|^2e^{-\varphi'+\psi})e^{t}dt\\
		\geq&(\int_{0}^{+\infty}e^{-q't+t}dt+\int_{-\infty}^{0}e^{t}dt)C_{F,\mathcal{I}(\varphi'+(q'-1)\psi)_o,\varphi'-\psi}(D)\\
		=&\frac{q'}{q'-1}C_{F,\mathcal{I}(\varphi'+(q'-1)\psi)_o,\varphi'-\psi}(D).
	\end{split}
\end{equation}
Thus, Proposition \ref{p} holds.
	\end{proof}

Let $F$ be a holomorphic function on pseudoconvex domain $D\subset\mathbb C^n$, which  contains the origin $o\in\mathbb C^n$. Let  $\psi$ be  a negative plurisubharmonic function on $D$ satisfying $\psi(o)=-\infty$, and let $\varphi'$ be a  plurisubharmonic function on $D$.

Define $c_o^{F}(\varphi';\psi):=\sup\{c\geq0:|F|^2e^{-\varphi'-(2c-1)\psi}$ is $L^1$ on a neighborhood of $o$$\}$. Especially, when $\psi=\varphi'$, $c_o^{F}(\varphi';\psi)$ degenerates to the jumping number $c_o^{F}(\varphi')$ (see \cite{JM13}).

\begin{Proposition}
\label{thm:effective}
	Assume that $\int_{D}|F|^2e^{-\varphi'}<+\infty$, and $c_o^{F}(\varphi';\psi)<+\infty$. Then for any $q'>1$ satisfying
	$$\frac{q'}{q'-1}>\frac{\int_{D}|F|^2e^{-\varphi'}}{C_{F,\mathcal{I}_+(\varphi'+(2c_o^{F}(\varphi';\psi)-1)\psi)_o,\varphi'-\psi}(D)},$$
	we have
	 $|F|^2e^{-\varphi'-(q'-1)\psi}$ is $L^1$ on a neighborhood of $o$.
\end{Proposition}
\begin{proof}
We give the proof of Proposition \ref{thm:effective} by using Propersition \ref{p}.

If $q'>2c_o^{F}(\varphi';\psi)$, then $C_{F,\mathcal{I}(\varphi'+(q'-1)\psi)_o,\varphi'-\psi}(D)\geq C_{F,\mathcal{I}_+(\varphi'+(2c_o^{F}(\varphi';\psi)-1)\psi)_o,\varphi'-\psi}(D)>0$. Combining inequality \eqref{eq:20210529a}, we have
\begin{equation}
	\label{eq:20210529f}
	\int_{D}|F|^2e^{-\varphi'}\geq\frac{q'}{q'-1}C_{F,\mathcal{I}_+(\varphi'+(2c_o^{F}(\varphi';\psi)-1)\psi)_o,\varphi'-\psi}(D).\end{equation}
Let $q'\rightarrow 2c_o^{F}(\varphi';\psi)$, we have inequality \eqref{eq:20210529f} also holds when $q'=2c_o^{F}(\varphi';\psi)$. Then we obtain that, if
$$\int_{D}|F|^2e^{-\varphi'}<\frac{q'}{q'-1}C_{F,\mathcal{I}_+(\varphi'+(2c_o^{F}(\varphi';\psi)-1)\psi)_o,\varphi'-\psi}(D),$$
then $q'<2c_o^{F}(\varphi';\psi)$,  $|F|^2e^{-\varphi'-(q'-1)\psi}$ is $L^1$ on a neighborhood of $o$.

Thus Proposition \ref{thm:effective} holds.	
\end{proof}

\begin{proof}[Proof of Theorem \ref{c:lp}]
$c_{o,p}^{F}(\varphi)<+\infty$ shows that $\varphi(o)=-\infty$.	We use Proposition \ref{thm:effective} to prove Theorem \ref{c:lp}.
	Replace $F$, $\psi$ and $\varphi'$ by $F_1$, $a\varphi$ and $\varphi_1+\varphi$, respectively, where $a>0$. It is clear that  $\mathcal{I}_+(\varphi_1+\varphi+(2c_o^{F_1}(\varphi_1+\varphi;a\varphi)-1)a\varphi)_o=\mathcal{I}_+(\varphi_1+2c_{o,p}^{F}(\varphi)\varphi)_o$. Note that
	$$C_{F_1,\mathcal{I}_+(\varphi_1+2c_{o,p}^{F}(\varphi)\varphi)_o,\varphi_1+(1-a)\varphi}(D)=\frac{1}{K_{\varphi,F,a}^{(p)}(o)}.$$
	  Let $q=(q'-1)a+1$, then it follows  from  Proposition \ref{thm:effective} that Theorem \ref{c:lp} holds.
\end{proof}

\section{Appendix}\label{Appe}
In this section, we prove that $C_{F,I,\Phi_1-\Phi_2}(D)=0$ if and only if $(F,o)\in I$. As $C_{F_1,\mathcal{I}_+(\varphi_1+2c_{o,p}^{F}(\varphi)\varphi)_o,\varphi_1+(1-a)\varphi}(D)=\frac{1}{K_{\varphi,F,a}^{(p)}(o)}$ and $F_1\not\in\mathcal{I}_+(\varphi_1+2c_{o,p}^{F}(\varphi)\varphi)_o$, therefore we get $K_{\varphi,F,a}^{(p)}(o)<+\infty$. Similarly, we have $K_{D,(1-a)\varphi}(o)<+\infty$.

It is clear that $(F,o)\in I\Rightarrow C_{F,I,\Phi_1-\Phi_2}(D)=0$. Thus, it suffices to prove that $C_{F,I,\Phi_1-\Phi_2}(D)=0\Rightarrow(F,o)\in I$.

By definition of $C_{F,I,\Phi_1-\Phi_2}(D)$, there exist holomorphic functions $\{F_j\}_{j\in \mathbb{N}^+}$ on $D$ such that $\lim_{j\rightarrow+\infty}\int_{D}|F_j|^2e^{-\Phi_1+\Phi_2}=0$ and $(F_j-F,o)\in I$. Fixed open subsets $D'$ and $D''$ of $D$ satisfying $D''\subset\subset D'\subset\subset D$. As $\Phi_1$ is plurisubharmonic on $D$, there exists a constant $C_1>0$ such that $\int_{D'}|F_j|^2e^{\Phi_2}\leq C_1\int_{D}|F_j|^2e^{-\Phi_1+\Phi_2}$ for any $j\in \mathbb{N}^+$, therefore
\begin{equation}
	\label{eq:210703a}\lim_{j\rightarrow+\infty}\int_{D'}|F_j|^2e^{\Phi_2}=0.\end{equation}

Since $\Phi_2$ is plurisubharmonic on $D$ and $\Phi_2\not\equiv-\infty$, there exists $s_0>0$ such that $\int_{D'}e^{-s\Phi_2}<+\infty$ for any $s\in[0,s_0)$.

Take a positive number $r$ satisfying $\frac{r}{1-r}\in(0,s_0)$, then $\int_{D'}e^{-\frac{r}{1-r}\Phi_2}<+\infty$. Following from  $|F_j|^{2r}$ is plurisubharmonic on $D$, we obtain that
\begin{displaymath}
	\begin{split}
		|F_j(z)|^{2r}&\leq C_2\int_{D'}|F_j|^{2r}\\
		&\leq C_2 (\int_{D'}|F_j|^{2}e^{\Phi_2})^{r}(\int_{D'}e^{-\frac{r}{1-r}\Phi_2})^{1-r}
	\end{split}
\end{displaymath}holds for any $z\in D''$ and $j\in \mathbb{N}^+$, where $C_2$ is a constant independent of $j$ and $z$. Then there exists a constant $C_3$ such that
\begin{equation}
	\label{eq:210703b}|F_j(z)|^{2}\leq C_3\int_{D'}|F_j|^{2}e^{\Phi_2}
\end{equation}
holds for any $z\in D''$ and $j\in \mathbb{N}^+$.
   Combining equality \eqref{eq:210703a} and inequality \eqref{eq:210703b}, we have  $\{F_j\}_{j\in \mathbb{N}^+}$ is uniformly convergent to $0$ on any compact subset of $D$. Hence, $I$ is closed under local uniform convergence (see \cite{G-R}) and $(F_j-F,o)\in I$ imply that $(F,o)\in I$.

\vspace{.1in} {\em Acknowledgements}. The authors would like to thank Shijie Bao for pointing out some typos, and giving a useful suggestion about the layout of this paper. The first named author was supported by NSFC-11825101, NSFC-11522101 and NSFC-11431013.

\bibliographystyle{references}
\bibliography{xbib}

\end{document}